\documentclass{article}
\title{The Generalized Point-Vortex Problem and Rotating Solutions to the Gross-Pitaevskii Equation on Surfaces of Revolution}
\author{Ko-Shin Chen \\Department of Mathematics, \\ Indiana University, Bloomington, IN 47405 \\ koshchen@indiana.edu}
\usepackage{amsmath,amsthm,amssymb}
\numberwithin{equation}{section}
\theoremstyle{plain}
\newtheorem{prop}{Proposition}[section]
\newtheorem{lem}{Lemma}[section]
\newtheorem{thm}{Theorem}[section]
\theoremstyle{remark}
\newtheorem{rk}{\bf Remark}[section]

\begin{document}
\maketitle
\begin{abstract}
We study the generalized point-vortex problem and the Gross-Pitaevskii equation on a surface of revolution. We find rotating periodic solutions to the generalized point-vortex problem, which have two two rings of n equally spaced vortices with degrees $\pm 1$. In particular we prove the existence of such solutions when the surface is longitudinally symmetric. Then we seek a rotating solution to the Gross-Pitaevskii equation having vortices that follow those of the point-vortex 
flow for $\varepsilon$ sufficiently small.
\end{abstract}

\section{Introduction}
In this paper we study the generalized point-vortex problem and the Gross-Pitaevskii equation on a surface of revolution, $\cal M$. Given $n$ points $\{a_i\}_{i=1}^n \subset \cal M$ and their associated degrees $\{d_i\}_{i=1}^n$, we consider the Hamiltonian system given by
\begin{equation} \label{ODE gen}
    d_i \frac{d}{dt} a_i = -\frac{1}{\pi} \nabla^\perp_{a_i} W(\mathbf a, \mathbf d) \mbox{ for } i = 1,2,...,n,
\end{equation}
where $W$ is a so-called renormalized energy depending on $\mathbf a = (a_1, a_2,...,a_n)$ and $\mathbf d =(d_1, d_2,...,d_n)$. The function $W$ involves a logarithmic interaction between the vortices and a precise definition is given below in \eqref{Wdefna}.
The system \eqref{ODE gen} arises in particular as the limit of the Gross-Pitaevskii equation
\begin{equation} \label{GP}
    iU_t = \Delta U + \frac{1}{\varepsilon^2} (1-|U|^2)U
\end{equation}
for $U: {\cal M} \times {\mathbb R} \rightarrow {\mathbb C}$, though in the plane and on the two-sphere it is more commonly associated with vortex motion for
the incompressible Euler equations. The Gross-Pitaevskii equation has been a fundamental model in studying superfluidity, Bose-Einstein condensation and nonlinear optics. Here $\cal M$ is compact, simply connected without boundary, and  $\Delta$ is the Laplace-Beltrami operator on $\cal M$. The dynamics of the flow \eqref{GP} preserve the Ginzburg-Landau energy
\begin{equation} \label{GL energy}
    E_\varepsilon(U) = \int_{\cal M} \frac{|\nabla U|^2}{2} +\frac{(1-|U|^2)^2}{4 \varepsilon^2}.
\end{equation}
Indeed if one makes the usual association of ${\mathbb C}$ with ${\mathbb R^2}$ and views $U$ as a map taking values in ${\mathbb R^2}$ then one can formally write \eqref{GP} as $U_t=\nabla^\perp E_\varepsilon(U).$
In the asymptotic regime $\varepsilon \ll 1$, the analysis of \eqref{GP} can be effectively carried out by tracking the motion of a finite number of vortices, which are zeros of $U$ with non-zero degrees. Furthermore, the role of $E_\varepsilon$ is replaced by $W$. In \cite{BBH} the importance of $W$ was first related to the stationary solution to \eqref{GP} in a planar domain with Dirichlet boundary conditions. From Theorem 4.3 in \cite{KC2}, the asymptotic motion law for vortices is governed by \eqref{ODE gen} where the points $\{a_i\}$ are viewed as vortices. In a planar domain or on a sphere, \eqref{ODE gen} is known as the classical point-vortex problem and has been studied extensively. But as far as we know, there has not been much work done when the problem is set on other manifolds. The primary goals of this article are two-fold:  To identify certain periodic solutions to \eqref{ODE gen} posed on surfaces of revolutions and then to establish the existence of corresponding periodic solutions to \eqref{GP}.
\\

To study \eqref{ODE gen} on $\cal M$, we will appeal to a result of \cite{B} where the author identifies $W$ on a Riemannian 2-manifold. For a compact, simply-connected surface without boundary, one can apply the Uniformization Theorem to assert the existence of a conformal map $h: {\cal M} \rightarrow {\mathbb R}^2 \bigcup\{\infty \}$, so that the metric $g$ is given by
\begin{equation}
e^{2f}(dx_1^2 + dx_2^2),
\end{equation}
for some smooth function $f$. Thus one may identify a vortex $a_i \in\cal M$ with a point $b_i=h(a_i)\in{\mathbb R}^2 \bigcup \{ \infty \}$. Writing
$\mathbf b = (b_1, b_2, ..., b_n)$ with associated degrees $\mathbf d = (d_1, d_2, ..., d_n)$, a result in \cite{B} identifies  the renormalized energy as
\begin{equation}\label{Wdefna}
    W(\mathbf b, \mathbf d) := \pi \sum_{i=1}^n d_i^2f(b_i) - \pi \sum_{i\neq j} d_i d_j \ln|b_i - b_j|.
\end{equation}
Then \eqref{ODE gen} can be rewritten as
\begin{equation} \label{GPV}
    d_i \dot{\mathbf b_i} = -e^{-2f(\mathbf b_i)} \left[\nabla^{\perp} f(\mathbf b_i) - 2 \sum_{j \neq i} d_i d_j\frac{(\mathbf b_i - \mathbf b_j)^{\perp}}{|\mathbf b_i- \mathbf b_j|^2}\right].
\end{equation}
We call \eqref{GPV} the generalized point-vortex problem on a simply connected Riemannian manifold. When the domain is ${\mathbb R}^2$, $f \equiv 0$ and $W$ reduces to the standard logarithm. When the domain is a sphere, $f$ is induced by the stereographic projection and $W$ actually is the sum of the logarithms of Euclidean distances between vortices. In both cases, \eqref{GPV} reduces to the classical point-vortex problem. A discussion of known result can be found in \cite{N}. One can also consider the case of a bounded planar domain with boundary conditions or the flat torus where the formulas for $W$ can be found in \cite{JS}, \cite{LX}, \cite{CJ1}, and \cite{CJ2}.
\\

In Section 2, we introduce a conformal map mapping from ${\mathbb R}^2 \bigcup \{\infty\}$ to $\cal M$ as in \cite{KC} and identify the explicit formula for $f$ in \eqref{Wdefna}. Then we find rotating periodic solutions to \eqref{ODE gen} having two rings, $C_\pm$, of $n$ equally spaced vortices with degrees $\pm 1$ such that the total degree is zero. In particular, in Proposition \ref{prop 1} we prove the existence of such solutions when $\cal M$  is longitudinally symmetric. Since $\cal M$ is compact without boundary, zero total degree
is necessary for making a connection with Gross-Pitaevskii vortices.
\\

On a sphere, the connection between Gross-Pitaevskii and point vortex dynamics has been studied in \cite{GS}. Following a similar argument in \cite{GS}, we generalize their results in Section 3 to the setting on $\cal M$ which is longitudinally symmetric. The approach is based on minimization of the Ginzburg-Landau energy \eqref{GL energy} subject to a momentum constraint. For any rotating periodic solution to \eqref{ODE gen} having two rings placed symmetrically, we construct in Theorems \ref{main 1} and \ref{main 2} a rotating solution to \eqref{GP} having vortices that follow those of the point-vortex flow for $\varepsilon$ sufficiently small.

\section{Generalized Point Vortex Motion}
In this section we study the generalized point vortex problem \eqref{GPV} on a surface of revolution.
Let $\cal M \subset {\mathbb R}$$^3$ be the surface obtained by rotating a regular curve

\[
        \gamma (s) = (\alpha(s),0,\beta(s)), \quad 0\leq s \leq l, \quad \alpha(s) > 0 \mbox{ for } s \neq 0, l.
\]
about the $Z$-axis, where $s$ is the arc length, i.e. $| \gamma '| = 1$.
Furthermore, make the assumptions
\[
     \alpha (0)  = \alpha(l) =\beta'(0) = \beta'(l) =  0,
\]
so that $\cal M$ is a smooth simply connected compact surface without boundary.
To parametrize $\cal M$, first we define $\mathbf P^\mu: {\cal S}^2 \rightarrow \cal M$ by
\begin{equation} \label{P mu}
\mathbf P^\mu(\theta, \phi) = (\alpha(S(\phi)) \cos \theta, \alpha(S(\phi)) \sin \theta, \beta(S(\phi))),
\end{equation}
where $S: [0, \pi] \rightarrow [0,l]$ satisfies
\begin{equation} \label{S ODE}
    S'(\phi) \sin \phi = \alpha (S(\phi)),
\end{equation}
and $(\theta, \phi)$ are spherical coordinates on ${\cal S}^2$ with $\phi$ corresponding to the angle made with $Z$-axis. In fact, \eqref{S ODE} makes the projection $\mathbf P^\mu$ conformal such that parameter values $(\theta, \phi)$ corresponding
to a point $\tilde p \in {\cal S}^2$ are mapped to parameter values $(\theta, S(\phi))$ corresponding to the point $p \in \cal M$, where $S(0) = 0$ and $S(\pi) =l$. 
Then let $\mathbf P^\nu: {\mathbb R}^2 \bigcup \{\infty\} \rightarrow {\cal S}^2$ to be the inverse of stereographic projection so that
\[
    \mathbf P^\nu (x,y) = (\theta, \phi)
\]
with
\begin{equation}\label{cos phi}
    \cos \phi = \frac{1-r^2}{1+r^2}, \quad r^2 = x^2+y^2,
\end{equation}
and $\theta \in [0,2\pi)$ is the polar angle of $(x,y)$ in ${\mathbb R}^2$. We parametrize $\cal M$ by defining $\mathbf P: {\mathbb R}^2 \bigcup \{\infty\} \rightarrow \cal M$ through
\begin{equation} \label{P}
    \mathbf P(x, y) = \mathbf P^\mu \circ \mathbf P^\nu (x,y) = (\alpha(S(\phi)) \cos \theta, \alpha(S(\phi)) \sin \theta, \beta(S(\phi))).
\end{equation}
Note that at this point, we view $\theta$ and $\phi$ as functions of $x$ and $y$.
Then the metric on $\cal M$ is given by
\begin{equation} \label{metric}
    g_{\cal M} = e^{2f} (dx^2 + dy^2),
\end{equation}
where
\begin{equation} \label{factor f}
    f = \ln \left( \frac{\alpha(S(\phi))}{r} \right).
\end{equation}
\\

Consider $2n$ vortices $\{\mathbf P_i\}_{i=1}^{2n}$ on $\cal M$ with degrees $\{d_i\}_{i=1}^{2n}$ and their projections $\{\mathbf p_i\}_{i=1}^{2n}$ on the plane defined by \eqref{P mu}-\eqref{P}. Using \eqref{S ODE}, \eqref{cos phi}, and \eqref{factor f} we derive
\begin{align} \label{grad f}
    \nabla f (\mathbf p) &= \frac{r}{\alpha(S(\phi))}\left[\alpha'(S(\phi))S'(\phi)\frac{\nabla \phi}{r} - \alpha(S(\phi))\frac{\mathbf p}{r^3}\right] \notag \\
    &=  \left[\alpha'(S(\phi))-1\right]\frac{\mathbf p}{r^2}.
\end{align}
Thus \eqref{GPV} can be rewritten as
\begin{equation} \label{ODE 2}
     d_i \dot{\mathbf p_i} = \frac{r_i^2}{\alpha_i^2} \left[ (1-\alpha_i') \frac{{\mathbf p_i}^{\perp}}{r_i^2} + 2 \sum_{j \neq i} d_i d_j\frac{(\mathbf p_i - \mathbf p_j)^{\perp}}{|\mathbf p_i- \mathbf p_j|^2}\right],
\end{equation}
where $\alpha_i = \alpha (S(\phi(\mathbf p_i)))$ and $\alpha'_i = \alpha' (S(\phi(\mathbf p_i)))$. Our goal here is to identify periodic rotating solutions to this system where the total degree is zero. Especially we will look for solutions whose orbits are circles on $\cal M$. We start with the case where $n=1$ and set $d_1 = -d_2 =1$.
In order to get a uniform rotational solution of \eqref{ODE 2}, we pursue the ansatz $\mathbf p_i = r_i (\cos \omega_0 t, \sin \omega_0 t)$ for $i = 1, 2$, where $\{r_i\}$ are constants, and $\omega_0$ is a constant to be determined. Plugging this into \eqref{ODE 2} we have
\begin{equation}
    \left\{\begin{array}{ll}
                 - r_1 \omega_0 = \frac{r_1^2}{\alpha_1^2} \left[ \frac{ 1-\alpha_1'}{r_1} - \frac{2}{r_1 - r_2} \right]\\
                 r_2 \omega_0 = \frac{r_2^2}{\alpha_2^2} \left[ \frac{1-\alpha_2'}{r_2} + \frac{2}{r_1 - r_2} \right].
                \end{array} \right.
\end{equation}
Hence $r_1$ and $r_2$ must satisfy
\begin{equation} \label{eq for r}
    -\frac{r_1}{\alpha_1^2} \left[ \frac{ 1- \alpha_1'}{r_1} - \frac{2}{r_1 - r_2} \right] = \frac{r_2}{\alpha_2^2} \left[ \frac{ 1-\alpha_2' }{r_2} + \frac{2}{r_1 - r_2} \right].
\end{equation}
Observe that if $\mathbf P_1$, $\mathbf P_2 \in \cal M$ satisfy
\[
    \alpha_1 = \alpha_2 \text{ and }  \alpha_1' = - \alpha_2',
\]
the equality \eqref{eq for r} holds automatically. In particular, when $\cal M$ is symmetric about plane $Z=\beta(\frac{l}{2})$ denoted by $Z_0$, \eqref{ODE 2} has a one parameter family of circular rotating solutions $\mathbf p_1$, $\mathbf p_2$ such that $\mathbf P_1$ and $\mathbf P_2$ on $\cal M$ are symmetrically located with respect to $Z_0$. From now on when considering  the symmetric situation, we will take $Z_0 = 0$ without loss of generality, i.e. $\cal M$ is symmetric about the $X$-$Y$ plane.
\\

Now we will look for rotating periodic solutions to \eqref{ODE 2} having two rings of $n$ equally spaced vortices with degree $1$ on one ring, $C_{+}$, and degree $-1$ on the other, $C_{-}$. Let $d_i = 1$ for $1 \leq i \leq n$, and $d_i = -1$ for $n+1 \leq i \leq 2n$. We assume that
\[
    \mathbf p_i = r_1 (\cos (\Theta_i + \omega_0t), \sin(\Theta_i + \omega_0t)) \mbox{ for } 1 \leq i \leq n
\]
and
\[
    \mathbf p_i = r_2 (\cos (\Theta_{i-n} + \omega_0t), \sin(\Theta_{i-n} + \omega_0t)) \mbox{ for } n+1 \leq i \leq 2n,
\]
where $\Theta_i = \frac{2\pi}{n}(i-1)$, $\{r_i\}$ are constants, and $\omega_0$ is a constant to be determined. Due to the symmetric vortex structure, it suffices to consider \eqref{ODE 2} for $i=1$ and $i=n+1$, i.e. equations for degree $1$ vortex and that for degree $-1$ vortex. When $i=1$, we have
\begin{align} \label{sum1}
    \sum_{j \neq i} d_i d_j & \frac{(\mathbf p_i - \mathbf p_j)^\perp}{|\mathbf p_i - \mathbf p_j|^2}
        =  \sum_{j=2}^n \frac{(\mathbf p_1 - \mathbf p_j)^\perp}{|\mathbf p_1 - \mathbf p_j|^2} -  \sum_{j=n+1}^{2n} \frac{(\mathbf p_1 - \mathbf p_j)^\perp}{|\mathbf p_1 - \mathbf p_j|^2} \notag \\
        = & \frac{n-1}{2}\frac{\mathbf p_1^\perp}{r_1^2} - \frac{1}{r_1 - r_2}\frac{\mathbf p_1^\perp}{r_1}  - \frac{\mathbf  p_1^\perp}{r_1} \sum_{j=2}^{\lfloor \frac{n+1}{2} \rfloor} \frac{2(r_1 - r_2 \cos \Theta_j)}{r_1^2 + r_2^2 - 2r_1 r_2 \cos \Theta_j} \notag \\
        & -  \left\{\begin{array}{ll}
                        \frac{1}{r_1 + r_2} \frac{\mathbf p_1^\perp}{r_1} & \mbox{ if $n$ is even,} \\
                        0 & \mbox{ if $n$ is odd.}
                \end{array} \right.
\end{align}
Similarly, when $i = n+1$, we have
\begin{align} \label{sum2}
    \sum_{j \neq i} d_i d_j & \frac{(\mathbf p_i - \mathbf p_j)^\perp}{|\mathbf p_i - \mathbf p_j|^2} \notag \\
        = & \frac{n-1}{2}\frac{\mathbf p_{n+1}^\perp}{r_2^2} - \frac{1}{r_2 - r_1}\frac{\mathbf p_{n+1}^\perp}{r_2}  - \frac{\mathbf  p_{n+1}^\perp}{r_2} \sum_{j=2}^{\lfloor \frac{n+1}{2} \rfloor} \frac{2(r_2 - r_1 \cos \Theta_j)}{r_1^2 + r_2^2 - 2r_1 r_2 \cos \Theta_j} \notag \\
        & -  \left\{\begin{array}{ll}
                        \frac{1}{r_1 + r_2} \frac{\mathbf p_{n+1}^\perp}{r_2} & \mbox{ if $n$ is even,} \\
                        0 & \mbox{ if $n$ is odd.}
                \end{array} \right.
\end{align}
Applying \eqref{ODE 2}, \eqref{sum1} and \eqref{sum2} we obtain for $1 \leq i \leq n$
\begin{equation} \label{rotating n vortex}
    \left\{\begin{array}{ll}
        -\omega_0 \mathbf p_i^\perp = \frac{\mathbf p_i^\perp}{\alpha_1^2} \left[ n - \alpha_1' -\frac{2 r_1}{r_1 - r_2} - r_1 Q(r_1, r_2)\right] \\
        \omega_0 \mathbf p_{n+i}^\perp = \frac{\mathbf p_{n+i}^\perp}{\alpha_2^2} \left[ n - \alpha_2' + \frac{2 r_2}{r_1 - r_2} - r_2 Q(r_2, r_1)\right],
    \end{array} \right.
\end{equation}
where
\begin{equation}
    Q(r_1,r_2) = \sum_{j=2}^{\lfloor \frac{n+1}{2} \rfloor} \frac{4(r_1 - r_2 \cos \Theta_j)}{r_1^2 + r_2^2 - 2r_1 r_2 \cos \Theta_j} +
    \left\{\begin{array}{ll}
        \frac{2}{r_1 + r_2} & \mbox { if $n$ is even,} \\
        0 & \mbox{ if $n$ is odd.}
    \end{array} \right.
\end{equation}
Thus $r_1$ and $r_2$ must satisfy
\begin{align} \label{eq for r general}
    - \frac{1}{\alpha_1^2} & \left[ n - \alpha_1' - \frac{2 r_1}{r_1 - r_2} - r_1 Q(r_1, r_2)\right] \notag \\
        & = \frac{1}{\alpha_2^2} \left[ n - \alpha_2' + \frac{2 r_2}{r_1 - r_2} - r_2 Q(r_2, r_1)\right]
\end{align}
Note that
\begin{equation} \label{rQ}
    r_1 Q(r_1, r_2) + r_2 Q(r_2, r_1) = 2n-2.
\end{equation}
From this we conclude again that for $\{\mathbf P_i$, $\mathbf P_{n+i}\}_{i=1}^n \subset \cal M$ satisfy
\[
    \alpha_1 = \alpha_2 \text{ and }  \alpha_1' = - \alpha_2',
\]
the equality \eqref{eq for r general} holds automatically. In particular, we have the following proposition: 
\begin{prop} \label{prop 1}
Suppose $\cal M$ is symmetric about the $X$-$Y$ plane. Then for any $n \geq 1$, there exists a rotating $2n$-vortex solution to \eqref{ODE 2} with orbits $C_{\pm}$ that are symmetric about the $X$-$Y$ plane.
\end{prop}

\section{Rotating Solutions to Gross-Pitaevskii}
In this section we will follow the basic methodology of \cite{GS} to show the existence of a rotating solution to \eqref{GP} having $2n$ vortices whose motion in the small $\varepsilon$ limit is governed by the rotating $2n$-vortex solution to \eqref{ODE 2} discussed in the previous section. For the rest of the article, we assume that $\cal M$ is symmetric about the $X$-$Y$ plane and parametrize it using the projection $\mathbf P^\mu: {\cal S}^2 \rightarrow \cal M$ defined through \eqref{P mu} and \eqref{S ODE}. Then the metric on $\cal M$ is given by
\begin{equation}
    g_\mu = e^{2 \mu} [\sin^2(\phi) d\theta^2 + d\phi^2],
\end{equation}
where
\[
    e^{2 \mu} = \frac{\alpha^2(S(\phi))}{\sin^2(\phi)},
\]
and S satisfies \eqref{S ODE}.
Given $p \in \cal M$, we denote by $\tilde{p}$ its projection on ${\cal S}^2$ via the inverse of $\mathbf P^\mu$. Let $\hat{p} \in \cal M$ be the reflection of $p$ about the $X$-$Y$ plane and $\hat{\tilde{p}} \in {\cal S}^2$ be the reflection of $\tilde{p}$ about the equator.
The symbols $\nabla$ and $\Delta$ refer to the gradient and the Laplace-Beltrami operator associated with the metric $g_\mu$. For $p \in \cal M$, the Green's function $G_{\cal M}:
{\cal M} \setminus \{p\} \rightarrow {\mathbb R}$ is the solution to
\[
    \Delta G_{\cal M} = \delta_{p} - \frac{1}{V_{\cal M}}.
\]
Here $V_{\cal M}$ is the
volume of $\cal M$ and the Laplace-Beltrami operator on $\cal M$ is
\[
    \Delta = e^{-2 \mu} \Delta_{{\cal S}^2},
\]
From Lemma 4.4 in \cite{S}, $G_{\cal M}$ can be expressed in terms of the Green's function on ${\cal S}^2$:
\begin{equation} \label{Green M}
    G_{\cal M}(x,p) = G_{{\cal S}^2}(\tilde{x},\tilde{p}) - \frac{1}{V_{\cal M}}
    q(\tilde{x}) + constant,
\end{equation}
where $q(\tilde{x})$ satisfies
\begin{equation}
    \Delta_{{\cal S}^2} q = e^{2 \mu} - \frac{1}{4 \pi} \int_{{\cal S}^2} e^{2\mu}. \notag
\end{equation}
For our purposes, all that is important here is that $q$ is a smooth function on $\cal M$.
Note that
\[
    G_{{\cal S}^2}(\tilde{x},\tilde{p}) = \frac{1}{2 \pi} \ln |\tilde{x} - \tilde{p}|,
\]
where $|\tilde{x} - \tilde{p}|$ is the chordal distance between $x$ and $p$. Hence for any fixed $p_1$, $p_2 \in {\cal M}$, the function $\Phi_0: {\cal M} \setminus \{p_1, p_2\} \rightarrow \mathbb R$ defined by
\begin{equation}\label{Def Phi_0}
    \Phi_0 (x) = \ln |\tilde x -  \tilde p_1| - \ln | \tilde x - \tilde p_2|
\end{equation}
satisfies
\[
    \Delta \Phi_0 = \delta_{p_1} - \delta_{p_2}.
\]
The following lemma will be used to construct a sequence of competitors when minimizing the Ginzburg-Landau energy.
\begin{lem}[cf. Lemma 3.1 of \cite{GS}] \label{Prop chi}
Consider $p_1, p_2 \in \cal M$ and fix $x_0 \in {\cal M} \setminus \{ p_1, p_2\}$.
Define $\chi : {\cal M} \setminus \{p_1, p_2\} \rightarrow \mathbb R$ by
\begin{equation}\label{chi}
\chi (x)=\int_\gamma \langle \nabla^\perp \Phi_0, \mathbf t \rangle,
\end{equation}
where $\gamma$ is any piecewise smooth simple curve in ${\cal M} \setminus\{p_1, p_2\}$ from $x_0$ to $x$, $\mathbf t$ is the unit tangent vector to $\gamma$, and $\Phi_0$ is given by \eqref{Def Phi_0}. Then
\begin{itemize}
\item[ \emph{(i)}] $\chi$ is well-defined up to an integer multiple of $2\pi$ for every $\tilde{x} \in {\cal M} \setminus \{ p_1, p_2\}$.
\item[ \emph{(ii)}] For $j = 1,2$, if $(\theta_j, \rho)$ are geodesic polar coordinates around the point $p_j$ and $B^\mu_j(r)\subset \cal M$ is a geodesic ball of radius $r$ centered at $p_j$, then $| \nabla (\theta_j-\chi) |={\cal O} (1)$ in $B^\mu_j(r)$ as $r \rightarrow 0$.
\item[ \emph{(iii)}] For any $p_1 \in \cal M$ not lying on the $X$-$Y$ plane, take $p_2=\hat{p}_1$ in \eqref{Def Phi_0}. Then, up to integer multiples of $2\pi$, $\chi$ is also symmetric with respect to the $X$-$Y$ plane.
\end{itemize}
\end{lem}

\begin{proof}
(i) Consider piecewise smooth simple curves $\gamma$ and $\gamma'$ from $x_0$ to $x$. Without loss of generality, we may assume that they do not intersect. Let $D \subset {\cal M}$ such that $\partial D = \gamma - \gamma'$. It is easy to see that if $D \bigcap \{p_1, p_2\} =\O$,
\[
    \int_{\gamma} \langle \nabla^\perp \Phi_0 , \mathbf t \rangle - \int_{\gamma'} \langle \nabla^\perp \Phi_0 , \mathbf t' \rangle =0\\
\]
Suppose that $D \bigcap \{p_1, p_2\} = \{p_1\}$. For $r>0$ such that $r$ is small enough, we have
\begin{align}
    0 & = \int_{D \setminus B^\mu_1(r)} \Delta \Phi_0 \notag \\
        & = \int_{\gamma} \langle \nabla \Phi_0 , \nu \rangle - \int_{\gamma'} \langle \nabla \Phi_0 , \nu \rangle - \int_{\partial B^\mu_1 (r)} \langle \nabla \Phi_0 , \nu \rangle \notag \\
         & = \int_{\gamma} \langle \nabla^\perp \Phi_0 , \mathbf t \rangle- \int_{\gamma'} \langle \nabla^\perp \Phi_0 , \mathbf t' \rangle- \int_{\partial B^\mu_1 (r)} \langle \nabla \Phi_0 , \nu \rangle. \notag
\end{align}
Thus
\begin{align} \label{welldef}
    \int_{\gamma} \langle \nabla^{\perp}  \Phi_0 , \mathbf t \rangle & - \int_{\gamma'} \langle \nabla^{\perp}  \Phi_0 , \mathbf t' \rangle
        = \int_{\partial B^{\mu}_1 (r)} \langle \nabla \Phi_0 , \nu \rangle \notag \\
        & = \int_{\partial B_1 (r)} \langle \nabla \ln |\tilde x - \tilde p_1| , \nu \rangle - \int_{B_1 (r)} \Delta \ln | \tilde x- \tilde p_2|,
\end{align}
where $B_1 (r) = (P^\mu)^{-1} (B^\mu_1 (r)) \subset {\cal S}^2$. Using \eqref{Green M} and \eqref{pde q} we obtain
\begin{align} \label{int for term2}
    \int_{B_1 (r)} \Delta \ln |\tilde x- \tilde p_2| = - \frac{1}{2} \int_{B_1 (r)} e^{-2 \mu} = {\cal O}(r).
\end{align}
The local geodesic polar coordinates around $p_1$ on $\cal M$ is given as
\begin{equation} \label{geo polar coord}
    G(\theta, \rho) d\theta^2 +  d \rho^2,
\end{equation}
where $G$ is a smooth function satisfies
\begin{equation} \label{lim G}
\lim_{\rho \rightarrow 0} \frac{G(\theta,\rho)}{\rho^2} =1.
\end{equation}
Now if $\rho=\rho (x)$ is the geodesic distance from $p_1$ to $x$, we may write
\begin{equation} \label{approx ln}
\ln |\tilde x - \tilde p_1| = \ln \rho -2\mu(\tilde p_1) + \cal O(\rho).
\end{equation}
Then applying the geodesic polar coordinates in $B^\mu_1 (r)$ we have
\begin{align} \label{int for term1}
   \int_{\partial B_1 (r)} \langle \nabla \ln |\tilde x - \tilde p_1| , \nu \rangle = &  \int_{\partial B^\mu_1 (r)} \langle \nabla [\ln \rho -2\mu(p_1) + O(\rho) ] , \nu \rangle \notag \\
        = & 2\pi + {\cal O} (r).
\end{align}
Combining \eqref{welldef}-\eqref{int for term1} and letting $r \rightarrow 0$ we deduce
\[
     \int_{\gamma} \langle \nabla^\perp \Phi_0 , \mathbf t \rangle - \int_{\gamma'} \langle \nabla^\perp \Phi_0 , \mathbf t' \rangle = 2\pi,
\]
i.e. $\chi$ is well-defined up to an integer multiple of $2\pi$ for every $x \in {\cal M} \setminus \{ p_1, p_2\}$.

(ii) Consider the geodesic polar coordinates around $p_1$ and denote $\theta_1$ by $\theta$. From \eqref{geo polar coord}, we have
\begin{align}
    \nabla \Phi_0 & = \frac{1}{G(\theta, \rho)} \frac{\partial \Phi_0}{\partial \theta} \partial_\theta + \frac{\partial \Phi_0}{\partial \rho} \partial_\rho \notag \\
     & = \frac{1}{\sqrt{G(\theta, \rho)}}  \frac{\partial \Phi_0}{\partial \theta} \mathbf e_\theta + \frac{\partial \Phi_0}{\partial \rho} \mathbf e_\rho.
\end{align}
Then
\[
    \nabla \chi = \nabla^\perp \Phi_0 = \left( \frac{1}{\rho} + {\cal O}(1) \right) \mathbf e_\theta - \frac{1}{\sqrt{G(\theta, \rho)}} {\cal O}(1) \mathbf e_\rho,
\]
and
\[
    \nabla \theta =  \frac{1}{G(\theta, \rho)} \partial_\theta = \frac{1}{\sqrt{G(\theta, \rho)}} \mathbf e_\theta.
\]
Thus from \eqref{lim G}, $|\nabla (\theta - \chi)| = {\cal O}(1)$ in $B^\mu_1 (r)$ as $r \rightarrow 0$. A similar argument applies to $B^\mu_2(r)$.

(iii) Since $p_1$ and $p_2$ on $\cal M$ are symmetric about the $X$-$Y$ plane, their projections $\tilde p_1$ and $\tilde p_2$ on ${\cal S}^2$ are also symmetric about the equator. Hence the argument in \cite{GS} is unchanged here.
\end{proof}

To obtain the convergence result as $\varepsilon \rightarrow 0$, we adapt the vortex-ball construction by Jerrard \cite{J} and Sandier \cite{S} to the setting on $\cal M$. Details of adjusting this technology to the setting of geodesic balls on a manifold can be found in Sec. 5 of \cite{CS}. Recall that for $\Omega \subset \cal M$ with smooth boundary $\partial \Omega$, the degree of a smooth function $u: \overline{\Omega} \rightarrow \mathbb C$ around $\partial \Omega$ is defined by
\[
    deg(u, \partial \Omega) = \frac{1}{2\pi} \int_{\partial \Omega} (iv, \partial_\tau v),
\]
where $u \neq 0$ on $\partial \Omega$, $v = \frac{u}{|u|}$, and $\tau$ is the unit tangent to $\partial \Omega$. We state below the adapted version from the corresponding result given in \cite{SS}, page 60.
\begin{lem} \label{ball construction}
Fix any $\zeta \in (0,1)$. Then there exists some
$\varepsilon_0(\zeta)>0$ such that for $0 < \varepsilon <
\varepsilon_0(\zeta)$, if $u: {\cal M} \rightarrow \mathbb C$
satisfies $E_\varepsilon(u) \leq \varepsilon^{\zeta-1}$, then there
is a finite collection of disjoint closed geodesic balls
$\{B^\mu_j(r_j)\}_{j=1}^{N_\varepsilon}$ so that
\begin{itemize}
\item[ \emph{(i)}] $\sum_{j=1}^{N_\varepsilon} r_j < C\varepsilon^{\frac{\zeta}{2}}$, where $C$ is a universal constant.
\item[ \emph{(ii)}] $\{x \in {\cal M}: ||u(\tilde{x})|-1| \geq \varepsilon^{\frac{\zeta}{4}} \} \subset \bigcup_{j=1}^{N_\varepsilon} B^\mu_j$.
\item[ \emph{(iii)}] We have
    \[
        \int_{\bigcup_{j=1}^{N_\varepsilon} B^\mu_j} \frac {|\nabla u|^2}{2} + \frac{(1-|u|^2)^2}{4 \varepsilon^2} \geq \pi D \left[ \left( 1-\frac{\zeta}{2} \right) |\ln \varepsilon| - \ln D - C \right],
    \]
where $D := \sum_{j=1}^{N_\varepsilon} |d_j|$ is assumed to be nonzero, and $d_j = deg(u, \partial B^\mu_j)$.
\item[ \emph{(iv)}] $D \leq C \frac{E_\varepsilon(u)}{\zeta |\ln \varepsilon|}$ with $C$ a universal constant.
\end{itemize}
\end{lem}

We will first show the existence of a rotating solution to the Gross-Pitaevskii equation having two vortices whose motion, as $\varepsilon$ approaches to $0$,  converges to the uniform rotating 2-vortex solution of the point-vortex problem \eqref{ODE 2}. In order to obtain a rotating solution of \eqref{GP}, we plug the
ansatz $U = u(R(\omega_\varepsilon t) x)$ into the equation, where $u: {\cal M} \rightarrow {\mathbb C}$,
$\omega_\varepsilon \in \mathbb R$ and $R(\Theta)$ is the rotation
matrix about the $Z$-axis given by
\begin{equation}
    \left(\begin{array}{ccc}
                 \cos \Theta & \sin \Theta & 0 \\
                - \sin \Theta & \cos \Theta & 0 \\
                0 & 0 & 1
                \end{array} \right) \mbox{ for } \Theta \in \mathbb R.
\end{equation}
Then $u$ must solve
\begin{equation} \label{GP2}
    - i \omega_\varepsilon \langle \nabla u , \tau \rangle = \Delta u + \frac{1}{\varepsilon} (1 - |u|^2) u,
\end{equation}
where for $p = (X,Y, Z) \in \cal M$, $\tau(p) = (-Y, X, 0) = \partial_\theta$. In fact, \eqref{GP2} is the Euler-Lagrange equation with $\omega_\varepsilon$ arising as a Lagrange multiplier for the following constrained minimization problem:
\begin{equation} \label{MP}
    \mbox{Minimize the Ginzburg-Langau energy } E_\varepsilon (u) \mbox{ for } u \in {\cal S}_{p_\varepsilon}
\end{equation}
where the admissible set is given by
\[
    {\cal S}_{p_\varepsilon} = \left \{ u \in H^1({\cal M},{\mathbb C}):P(u) = p_\varepsilon, \mbox{ and } u(x) = u(\hat{x}) \mbox{ for all } x \in {\cal M}  \right \},
\]
and the momentum $P$ of $u$ is defined as
\[
    P(u) = \mbox{Im} \int_{\cal M} u^* \langle \nabla u , \tau \rangle.
\]
The existence of a minimizer to \eqref{MP} can be shown by the direct method (\cite{GS}, Proposition 4.3), provided ${\cal S}_{p_\varepsilon}$ is nonempty and this minimizer $u_\varepsilon \in {\cal S}_{p_\varepsilon}$ will satisfy \eqref{GP2}.
\\

Recall that $\cal M$ is a surface of revolution generated by a curve
\[
    \gamma(s) = (\alpha(s), 0, \beta(s))
\]
with arc length $l$. In the next lemma we will, through constructions, prove that given a value $p$ defined in terms of $\alpha$ and $l$, there is a sequence of minimizers satisfying a certain energy bound with momenta $p_\varepsilon$ converging to $p$.
\begin{lem} \label{construct v}
Fix $s_1 \in (0, \frac{l}{2})$ and let $p = 2 \pi \int_{s_1}^{l-s_1}
\alpha(s) ds$. Then there exists a sequence
$\{p_\varepsilon\}_{\varepsilon >0}$ converging to $p$ as
$\varepsilon \rightarrow 0$ and a corresponding sequence of
minimizers  $\{u_\varepsilon\}$ of $E_\varepsilon$ in ${\cal
S}_{p_\varepsilon}$ such that $E_\varepsilon (u_\varepsilon) \leq 2
\pi | \ln \varepsilon| + {\cal O} (1)$.
\end{lem}
\begin{proof}
It is sufficient to construct a sequence of functions $\{ v_\varepsilon \} \subset H^1({\cal M}; {\mathbb C})$ such that each $v_\varepsilon$ is symmetric about the $X$-$Y$ plane, $E_\varepsilon (v_\varepsilon)$ satisfies the desired logarithmic bound and
\[
    P(v_\varepsilon) \rightarrow p \quad \mbox{as} \quad \varepsilon \rightarrow 0.
\]
Note that a minimizer always exists for an nonempty ${\cal S}_{p_\varepsilon}$. Then taking $p_\varepsilon = P(v_\varepsilon)$ the lemma is proved. The construction of $\{v_\varepsilon\}$ is based on Proposition 4.4 in \cite{GS}. Given $s_1 \in (0, \frac{l}{2})$, let $x_1 = (\alpha(s_1), 0, \beta(s_1)) \in {\cal M}$ and $B^\mu_1(r)$ and $\hat{B}^\mu_1(r)$ be the geodesic balls with radius $r$ centered at $x_1$ and $\hat x_1$. Fix $\varepsilon>0$ small enough such that $B^\mu_1(\varepsilon + \varepsilon^2) \bigcap \hat{B}^\mu_1(\varepsilon + \varepsilon^2) = \O$ and we may define $w_\varepsilon$ with the local geodesic polar coordinates $(\theta, \rho)$ around $x_1$ through
\begin{equation}
    w_\varepsilon (x) =
            \left\{\begin{array}{ll}
                 \frac{\rho}{\varepsilon} e^{i \theta} & \mbox{if } x \in B^\mu_1(\varepsilon) \\
                e^{i (\frac{\varepsilon + \varepsilon^2 - \rho}{\varepsilon^2}\theta + \frac{\rho - \varepsilon}{\varepsilon^2}\chi)} & \mbox{if } x \in  B^\mu_1(\varepsilon + \varepsilon^2) \setminus B^\mu_1(\varepsilon),
            \end{array} \right.
\end{equation}
where $\chi: {\cal M} \setminus \{x_1, \hat x_1\} \rightarrow \mathbb R$ is given by \eqref{chi} and $\theta$ is chosen so that $\theta = \chi$ at some points on $\partial  B^\mu_1(\varepsilon)$. Now we set

\begin{equation}
    v_\varepsilon (x) =
            \left\{\begin{array}{ll}
                 e^{i \chi} & \mbox{if } x \in {\cal M} \setminus (B^\mu_1(\varepsilon + \varepsilon^2) \bigcup \hat{B}^\mu_1(\varepsilon + \varepsilon^2)) \\
                w_\varepsilon (x) & \mbox{if } x \in B^\mu_1(\varepsilon + \varepsilon^2) \\
                w_\varepsilon (\hat{x}) & \mbox{if } x \in \hat{B}^\mu_1(\varepsilon + \varepsilon^2)
            \end{array} \right.
\end{equation}
From Lemma \ref{Prop chi}, $v_\varepsilon$ is well-defined and symmetric about the plane $Z = \beta(\frac{l}{2})$.\\
\\
{\bf - Estimate of $E_\varepsilon (v_\varepsilon)$}\\
Let $r = \varepsilon + \varepsilon^2$. Using the fact that $\Delta \Phi_0 =0 $ in ${\cal M} \setminus B^\mu_1(r) \bigcup \hat B^\mu_1(r)$ we derive
\begin{align} \label{e outside balls}
    \int_{{\cal M} \setminus B^\mu_1(r) \bigcup \hat B^\mu_1(r)} |\nabla v_\varepsilon|^2 = & \int_{{\cal M} \setminus B^\mu_1(r) \bigcup \hat B^\mu_1(r)} |\nabla^\perp \Phi_0|^2 \notag \\
    = & \int_{{\cal M} \setminus B^\mu_1(r) \bigcup \hat B^\mu_1(r)} |\nabla \Phi_0|^2 \notag \\
    = & - \int_{\partial B^\mu_1(r)} \Phi_0 \langle \nabla \Phi_0, \nu \rangle -  \int_{\hat B^\mu_1(r)} \Phi_0 \langle \nabla \Phi_0, \nu \rangle.
\end{align}
From \eqref{Def Phi_0} we have
\begin{align}
    - & \int_{\partial B^\mu_1(r)} \Phi_0 \langle \nabla \Phi_0, \nu \rangle \notag \\
    = & -\int_{\partial B_1(r)} (\ln |\tilde x - \tilde x_1| -  \ln |\tilde x - \hat{\tilde x}_1|)( \langle \nabla \ln |\tilde x - \tilde x_1|, \nu \rangle -  \langle \nabla \ln |\tilde x - \hat{\tilde x}_1| , \nu \rangle) \notag.
\end{align}
Using \eqref{approx ln} for $x \in \partial B^\mu_1(r)$ we obtain
\begin{equation} \label{bdd int1}
\ln |\tilde x - \tilde x_1| = \ln r -2\mu(\tilde x_1) + {\cal O}(r),
\end{equation}
\begin{equation}
    \langle \nabla \ln |\tilde x - \tilde x_1| , \nu \rangle = \frac{1}{r} + {\cal O}(1).
\end{equation}
Moreover,
\begin{equation}
    \ln |\tilde x - \hat{\tilde x}_1| = \ln |\tilde x_1 - \hat{\tilde x}_1| - 2\mu(\tilde x_1)+ {\cal O}(r),
\end{equation}
\begin{equation} \label{bdd int4}
    \nabla \ln |\tilde x - \hat{\tilde x}_1| = {\cal O}(1).
\end{equation}
Then combining \eqref{bdd int1}-\eqref{bdd int4} gives
\begin{equation} \label{e outside balls 2}
     - \int_{\partial B^\mu_1(r)} \Phi_0 \langle \nabla \Phi_0, \nu \rangle = 2 \pi \ln r + 2\pi \ln |\tilde x_1 - \hat{\tilde x}_1| + {\cal O}(r).
\end{equation}
The integral over $\partial \hat B^\mu_1(r)$ is treated in a similar way. Note that $r = \varepsilon + \varepsilon^2$. Thus by \eqref{e outside balls} and \eqref{e outside balls 2} we have
\begin{equation} \label{E1}
    \frac{1}{2} \int_{{\cal M} \setminus B^\mu_1(r) \bigcup \hat B^\mu_1(r)} |\nabla v_\varepsilon|^2 = 2 \pi |\ln \varepsilon| + 2\pi \ln |\tilde x_1 - \hat{\tilde x}_1| + o(1).
\end{equation}
Next we calculate the energy contribution inside balls. For $x \in B^\mu_1(\varepsilon)$, $v_\varepsilon = w_\varepsilon$ and
\[
    \nabla w_\varepsilon = \frac{1}{\sqrt{G(\theta, \rho)}} \frac{i\rho}{\varepsilon} e^{i\theta} \mathbf e_\theta + \frac{1}{\varepsilon} e^{i \theta} \mathbf e_\rho,
\]
where $\lim_{\rho \rightarrow 0} \frac{G(\theta, \rho)}{\rho^2} =1$. Hence
\begin{align} \label{E2}
    \int_{B^\mu_1(\varepsilon)} & \frac{|\nabla v_\varepsilon|^2}{2} + \frac{(1-|v_\varepsilon|^2)^2}{4 \varepsilon^2} \notag \\
        & = \frac{1}{4\varepsilon^2}\int_0^{2\pi}\int_0^\varepsilon \left[ \frac{2 \rho^2}{G(\theta, \rho)}+2 + \left( 1 - \frac{\rho^2}{\varepsilon^2} \right)^2 \right] \sqrt{G(\theta, \rho)} d\rho d\theta \notag \\
        & = {\cal O}(1).
\end{align}
In $B^\mu_1(\varepsilon + \varepsilon^2) \setminus  B^\mu_1(\varepsilon)$,  we have
\[
    |\nabla v_\varepsilon|^2 = \frac{1}{\varepsilon^4} |(\chi - \theta) \nabla \rho + (\rho - \varepsilon) \nabla (\chi - \theta) + \varepsilon^2 \nabla \theta|^2.
\]
Since $\theta$ here is chosen so that $\theta = \chi$ at some points on $\partial B^mu_1(\varepsilon)$, by Lemma \ref{Prop chi}, $|\nabla (\chi - \theta)| = {\cal O} (1)$ and $|\chi - \theta| = {\cal O} (\varepsilon)$. Then
\begin{equation} \label{E3}
     \int_{B^\mu_1(\varepsilon + \varepsilon^2) \setminus  B^\mu_1(\varepsilon)} \frac{|\nabla v_\varepsilon|^2}{2} = {\cal O}(\varepsilon).
\end{equation}
A similar estimate also holds in $\hat B^\mu_1(\varepsilon + \varepsilon^2)$. Finally combining \eqref{E1}, \eqref{E2} and \eqref{E3} we obtain
\begin{equation} \label{E upper bdd}
    E_\varepsilon (v_\varepsilon) = 2 \pi |\ln \varepsilon| + {\cal O}(1).
\end{equation}
\\
{\bf - Estimate of $P(v_\varepsilon)$}\\
This part follows exactly as in \cite{GS} except the $\phi$-coordinate on ${\cal S}^2$ is replaced by the $s$-coordinate on $\cal M$. We include the argument for the sake of completeness. 
Let $C_s = \{ (\alpha(s) \cos \theta, \alpha(s) \sin \theta,
\beta(s)):\theta \in [0,2\pi] \} \subset \cal M$ be the circle
corresponding to the arc length value $s$. We decompose the momentum
$P(v_\varepsilon)$ into two parts:
\begin{equation}
    P(v_\varepsilon) = \mbox{Im} \int_{\{\bigcup C_s: s \not \in
    Z\}} v^*_\varepsilon \langle \nabla v_\varepsilon , \tau \rangle
    + \mbox{Im} \int_{\{\bigcup C_s: s \in
    Z\}} v^*_\varepsilon \langle \nabla v_\varepsilon , \tau
    \rangle,
\end{equation}
where
\[
    Z = \{ s \in [0,l]: C_s \bigcap (B^\mu_1(\varepsilon + \varepsilon^2) \bigcup \hat B^\mu_1(\varepsilon + \varepsilon^2))\neq
    \O\}.
\]
Note that $Z$ can by expressed as a union of two disjoint intervals,
\[
    Z = (s_1', s_1'') \bigcup (l-s_1'', l-s_1'),
\]
and in $\{\bigcup C_s: s \not \in Z\}$,
\[
    \mbox{Im } v^*_\varepsilon \langle \nabla v_\varepsilon , \tau \rangle = \frac{\partial \chi}{\partial \theta}.
\]
Therefore using the facts that
\[
    \int_0^{2\pi} \frac{\partial \chi}{\partial \theta}d\theta =0
    \mbox{ for } s \in (0,s_1') \bigcup (l - s_1', l),
\]
and
\[
    \int_0^{2\pi} \frac{\partial \chi}{\partial \theta}d\theta =
    2\pi
    \mbox{ for } s \in (s_1'',l - s_1''),
\]
we have
\begin{align}
    \mbox{Im} \int_{\{\bigcup C_s: s \not \in
    Z\}} v^*_\varepsilon \langle \nabla v_\varepsilon , \tau \rangle
    & = \int_{s_1''}^{l- s_1''} \alpha(s) \int_0^{2\pi}\frac{\partial \chi}{\partial
    \theta}d\theta ds \notag\\
    & = 2\pi \int_{s_1''}^{l-s_1''} \alpha(s)ds. \notag \\
    & = 2\pi \int_{s_1}^{l-s_1} \alpha(s)ds + {\cal O}(\varepsilon)
    \notag \\
    & = p + {\cal O}(\varepsilon).
\end{align}
For the second part of $P(\varepsilon)$, since $|\{\bigcup C_s: s
\in Z\}| = {\cal O}(\varepsilon)$, using \eqref{E upper bdd} we derive
\begin{align}
|\mbox{Im} \int_{\{\bigcup C_s: s \in Z\}} v^*_\varepsilon \langle \nabla v_\varepsilon , \tau \rangle| & \leq \int_{\{\bigcup C_s: s \in Z\}} |\nabla v_\varepsilon| \notag \\
    & \leq |\{\bigcup C_s: s \in Z\}|^{\frac{1}{2}} \left[\int_{\cal M} |\nabla v_\varepsilon|^2\right]^{\frac{1}{2}} \notag \\
    & = o(1).
\end{align}
Hence $P(v_\varepsilon) = p + o(1)$ and the lemma is proven.
\end{proof}
With Lemma \ref{ball construction} and Lemma \ref{construct v} we
may extend the following result stated in \cite{GS} to the setting
on $\cal M$.
\begin{thm} \label{main 1}
Let $p_\pm(t)$ be any rotating 2-vortex solution to \eqref{ODE 2} on $\cal M$ with circular orbits denoted by $C_\pm$ that are symmetric about the plane the $X$-$Y$ plane.
Then for all positive $\varepsilon$ sufficiently small, there exists
a solution to \eqref{GP} of the form
$U_\varepsilon (x,t) = u_\varepsilon(R(\omega_\varepsilon t)x)$ where $u_\varepsilon: {\cal M} \rightarrow \mathbb C$ minimizes \eqref{MP} rotating about the
$Z$-axis. Furthermore, there exists a finite collection of disjoint
balls ${\cal B}_\varepsilon$ in $\cal M$, including two balls
$B_\pm^\varepsilon$, such that
\begin{itemize}
\item[ \emph{(i)}] $u_\varepsilon$ is symmetric about the plane the $X$-$Y$ plane and $E_\varepsilon(u_\varepsilon) \leq 2\pi | \ln \varepsilon | + {\cal O} (1)$ as $\varepsilon \rightarrow 0$.
\item[ \emph{(ii)}] The balls $B_\pm^\varepsilon$ and their centers $p_\pm^\varepsilon$ are symmetric about the $X$-$Y$ plane, and their common radius $r_\varepsilon$ converges to zero as $\varepsilon \rightarrow 0$.
\item[ \emph{(iii)}] $deg(u_\varepsilon,\partial B_\pm^\varepsilon)=\pm 1$ and $deg(u_\varepsilon,\partial B)=0$ for any $B \in {\cal B}_\varepsilon \setminus \{ B_\pm^\varepsilon \}$.
\item[ \emph{(iv)}] $| u_\varepsilon (x)|>1/2$ for $x \in {\cal M} \setminus {\cal B}_\varepsilon$ and $| {\cal B}_\varepsilon |=o(1)$ as $\varepsilon \rightarrow 0$.
\item[ \emph{(v)}] The circular orbits $C_\pm^\varepsilon$ associated with the rotation of $p_\pm^\varepsilon$ approach $C_\pm$ as $\varepsilon \rightarrow 0$.
\end{itemize}
\end{thm}
The theorem above indicates that if one begins with a 2-vortex solution to the system of ODE's \eqref{ODE 2} where the vortices are located at heights corresponding to the arc length parameter $s=s_1$ and $l-s_1$ for some $0<s_1<\frac{l}{2}$, then one picks $p_\varepsilon$ as in the construction so that it converges to $p$ given by the formula $2 \pi \int_{s_1}^{l-s_1} \alpha(s) ds$ as in Lemma \ref{construct v}.
Since the method of proving Theorem \ref{main 1} in \cite{GS} is
independent of the geometry, the same argument can be applied here.
Furthermore, Proposition \ref{prop 1} indicates that there
exists a rotating $2n$-vortex solution to \eqref{ODE 2} when $\cal
M$ is symmetric about the $X$-$Y$ plane. Then we define
\[
    S^n_{p_\varepsilon} = \{ u \in H^1({\cal M}; {\mathbb C}): P(u) = p_\varepsilon, u(x) = u(\hat x), 
\]
\[
\mbox{and } u(x) = u(R\left( \frac{2\pi}{n}\right)x) \mbox{ for all } x \in {\cal  M}\}.
\]
In the same spirit as the proof of Theorem \ref{main 1}, we have the following generalized result:
\begin{thm} \label{main 2}
For $n \geqslant 1$, consider a rotating $2n$-vortex solution to \eqref{ODE 2} with circular orbits denoted by $C_\pm$ that are symmetric about the $X$-$Y$ plane. Then for all positive $\varepsilon$ sufficiently small, there exists a solution to \eqref{GP} of the form $U_\varepsilon(x,t) = u_\varepsilon(R(\omega_\varepsilon t)x)$ with $u_\varepsilon \in S_{p_\varepsilon}^n$. Furthermore, there exists a finite collection of disjoint balls ${\cal B}_\varepsilon$, including $2n$ balls $\{ B_{j,\pm}^\varepsilon \}_{j=1}^n$ such that
\begin{itemize}
\item[ \emph{(i)}] $E_\varepsilon(u_\varepsilon) \leqslant 2\pi n | \ln \varepsilon |+{\cal O}(1)$ as $\varepsilon \rightarrow 0$.
\item[ \emph{(ii)}] The balls $B_{j,\pm}^\varepsilon$ and their centers $p_{j,\pm}^\varepsilon$ are symmetric about the $X$-$Y$ plane and invariant under a $\frac{2 \pi}{n}$ rotation about the $Z$-axis. Furthermore, their common radius $r_\varepsilon$ converges to zero as $\varepsilon \rightarrow 0$.
\item[ \emph{(iii)}] $deg(u_\varepsilon,\partial B_{j,\pm}^\varepsilon)=\pm 1$ and $deg(u_\varepsilon,\partial B)=0$ for any $B \in {\cal B}_\varepsilon \setminus \{ B_{j,\pm}^\varepsilon \}_{j=1}^n$.
\item[ \emph{(iv)}] $| u_\varepsilon(x) |>1/2$ for $x \in {\cal M} \setminus {\cal B}_\varepsilon$ and $| {\cal B}_\varepsilon | \rightarrow 0$ as $\varepsilon \rightarrow 0$.
\item[ \emph{(v)}] The circular orbits $C_\pm^\varepsilon$ associated with the rotation of $p_\pm^\varepsilon$ approach $C_\pm$ as $\varepsilon \rightarrow 0$.
\end{itemize}
\end{thm}

\begin{rk}
From \eqref{eq for r general} in section 2, there can be a rotating $2n$-vortex solution to \eqref{ODE 2} with orbits $C_{\pm}$ on some non-symmetric surfaces of revolution. Note that Theorem 4.3 in \cite{KC2} indicates that for finite times there exists s solution to \eqref{GP} with vortices following these orbits. It would be interesting to see if one can construct a periodic solution to \eqref{GP} with vortices following these orbits all the time.
\end{rk}

\section*{Acknowledgment}
I would like to express my appreciation and thanks to my adviser, Professor Peter Sternberg, for his invaluable advice on this paper.

\end{document}